\documentclass[11pt]{article}

\usepackage[utf8]{inputenc}
\usepackage{amsmath,amssymb,amsthm,amsfonts}
\usepackage{geometry}
\usepackage{authblk}
\usepackage{enumitem}

\numberwithin{equation}{section}

\theoremstyle{plain}
\newtheorem{theorem}{Theorem}[section]
\newtheorem{proposition}[theorem]{Proposition}
\newtheorem{lemma}[theorem]{Lemma}
\newtheorem{corollary}[theorem]{Corollary}
\theoremstyle{definition}
\newtheorem{definition}[theorem]{Definition}
\theoremstyle{remark}
\newtheorem{remark}[theorem]{Remark}

\DeclareMathOperator{\GL}{GL}
\DeclareMathOperator{\LGL}{LGL}
\DeclareMathOperator{\LM}{LM}
\DeclareMathOperator{\diag}{diag}
\DeclareMathOperator{\St}{St}
\newcommand{\Lc}{L_K(R_2)}
\newcommand{\Words}{\{e,f\}^*}
\newcommand{\Cantor}{\{e,f\}^{\mathbb N}}
\newcommand{\LGLtwo}{\LGL_2(K)}
\newcommand{\Mtwo}{\mathcal M_2(K)}

\title{Matrix generators for the unit groups of $L_K(1,d)$}
\author[1]{Huynh Viet Khanh}
\author[1]{Vo Hoang Thanh}
\affil[1]{Department of Mathematics and Informatics, Ho Chi Minh City University of Education, Ho Chi Minh City, Vietnam\\
\texttt{khanhhv@hcmue.edu.vn}; \texttt{kv.thanhhv@hcmue.edu.vn}}
\date{}

\begin{document}
\maketitle

\begin{abstract}
Let $K$ be a field and put $L_d=L_K(R_d)\cong L_K(1,d)$. Ordered leaf sets in the rooted $d$-ary tree determine copies of general linear groups over $K$ inside $L_d^\times$. We prove that these copies generate $L_d^\times$ for every $d\geq2$. In the binary case, $L_2^\times=\langle 1+eaf^*,1+fbe^*:a,b\in L_2\rangle$. We characterize finite generation of $L_d^\times$, determine the subgroup represented by monomial matrices, and embed $\GL_\infty(K)$ in $L_2^\times$. Over a finite field, finite presentability of $L_d^\times$ is equivalent to finite generation of the unstable $K_2$-group $K_2(n,L_d)$ for every $n=1+r(d-1)\geq5$, where $r\geq0$; we also compute $K_2(L_d)$.
\end{abstract}

\section{Introduction}

For $d\geq2$, let $R_d$ be the rose with one vertex and $d$ loops, and put $L_d=L_K(R_d)\cong L_K(1,d)$. In the terminology of Higman--Thompson groups, the leaves of a finite full subtree of the rooted $d$-ary tree form a finite basis obtained from the one-element basis by simple expansions; see \cite[Section~1]{Pardo}. If $C=(\gamma_1,\ldots,\gamma_n)$ is an ordering of these leaves, then the elements $\gamma_i\gamma_j^*$ form a complete system of matrix units. Thus $C$ determines a unital copy of $M_n(K)$ in $L_d$ and an isomorphism $\Theta_C:M_n(L_d)\to L_d$. Permutation matrices yield the usual copy of the Higman--Thompson group $G_{d,1}$ in $L_d^\times$. The isomorphisms $\Theta_C$ are special cases of the matrix-ring isomorphisms studied in \cite{AbramsAnhPardo}.

The binary algebra has a further universality property. Write $L=L_K(R_2)$ and denote the two loops by $e$ and $f$. Brownlowe and S\o rensen proved that every Leavitt path algebra of a countable graph over $K$ embeds in $L$ \cite[Theorem~4.1]{BrownloweSorensen}. Freeland defined the Leavitt general linear group as the subgroup of $L^\times$ generated by the copies of $\GL_n(K)$ associated with finite binary trees \cite[Definitions~4.3.9--4.3.10]{Freeland}. He asked whether this group coincides with $L^\times$ \cite[p.~173]{Freeland}.

We use the following convenient form of Freeland's construction. For ordered $d$-ary leaf sets $A=(\alpha_1,\ldots,\alpha_n)$ and $B=(\beta_1,\ldots,\beta_n)$ and for $M=(m_{ij})\in\GL_n(K)$, set
$$
u(A,M,B)=\sum_{i,j=1}^n m_{ij}\alpha_i\beta_j^*.
$$
Let $\LGL_d(K)$ be the subgroup generated by the elements $u(A,M,B)$. In the binary case, the factorization $u(A,M,B)=u(A,M,A)u(A,I_n,B)$ identifies $\LGL_2(K)$ with Freeland's group, since the second factor lies in the standard copy of $V=G_{2,1}$.

Our first main result is $\LGL_d(K)=L_d^\times$ for every field $K$ and every $d\geq2$. The elements $u(A,M,B)$ are not closed under multiplication, so the assertion concerns the group they generate. The proof has two ingredients. First, if $\rho$ and $\sigma$ are prefix-incomparable, then $1+\rho a\sigma^*\in\LGL_d(K)$ for every $a\in L_d$. It follows that $\Theta_C(E_n(L_d))\leq\LGL_d(K)$. Second, $L_d$ is a purely infinite simple $\mathrm{GE}$-ring, and
$$
(L_d^\times)_{\mathrm{ab}}\cong K_1(L_d)
\cong K^\times/(K^\times)^{d-1}.
$$
The map induced by the scalar inclusion $K^\times\to L_d^\times$ is surjective onto $(L_d^\times)_{\mathrm{ab}}$, so $\GL_n(L_d)=E_n(L_d)D_n(K)$. The isomorphism $\Theta_C$ sends the elementary factor into $\LGL_d(K)$, while the diagonal factor has the form $u(C,D,C)$. For $d=2$, we have $K_1(L)=0$, and
$$
L^\times=\left\langle 1+eaf^*,\ 1+fbe^*:a,b\in L\right\rangle.
$$

In the binary case, the units represented by monomial matrices form a subgroup $\mathcal M_2(K)$, and
$$
\mathcal M_2(K)\cong
C_{\mathrm{lc}}(\{e,f\}^{\mathbb N},K^\times)\rtimes V,
$$
where $V$ acts by pullback. We also use the matrix units $(e^{i-1}f)(e^{j-1}f)^*$ to embed $\GL_\infty(K)$ in $L^\times$. Finally, $L_d^\times$ is finitely generated if and only if $K$ is finite. Necessity follows by comparing coefficients in the basis associated with a choice of special edge; sufficiency follows from finite generation of a suitable elementary group.

Finite presentation requires a separate argument. If $K=\mathbb F_q$ and $n=1+r(d-1)\geq5$, where $r\geq0$, then
$$
L_d^\times\text{ is finitely presented}
\text{ if and only if }
K_2(n,L_d)\text{ is finitely generated}.
$$
The same condition characterizes finite presentability of $L_d^\times/K^\times$ and $E_n(L_d)$. Here $K_2(n,L_d)$ is the kernel of $\St_n(L_d)\to E_n(L_d)$; for $n\geq5$, it is naturally isomorphic to the Schur multiplier $H_2(E_n(L_d),\mathbb Z)$.

The stable $K_2$-group is
$$
K_2(L_d)\cong
\{\lambda\in\mathbb F_q^\times:\lambda^{d-1}=1\}.
$$
This group is finite and vanishes when $K=\mathbb F_2$ or $d=2$. The calculation does not determine $K_2(n,L_d)$, since the usual stability theorem does not apply: the Bass stable rank of $L_d$ is infinite \cite[Theorem~2.8(2)]{AraPardoStableRank}. Thus finite presentability is reduced to finite generation of $K_2(n,L_d)$.

Section~\ref{sec:leaf-sets} introduces leaf sets and the associated matrix units. Sections~\ref{sec:leaf-matrix} and \ref{sec:full-unit-group} prove that $\LGL_2(K)=L^\times$. Section~\ref{sec:finite-generation} establishes the finite-generation criterion, and Section~\ref{sec:subgroups} treats the monomial subgroup and $\GL_\infty(K)$. Section~\ref{sec:higher-arity} extends the generation results to $d$ loops. Section~\ref{sec:finite-presentability} relates finite presentability to unstable $K_2$.

\section{Leaf sets and matrix units}\label{sec:leaf-sets}

Let $\Words$ be the free monoid on $\{e,f\}$, with identity $1$. We identify its elements with the finite paths in $R_2$. The rooted binary tree $T_2$ has vertex set $\Words$, root $1$, and edges joining each word $\alpha$ to its two children $\alpha e$ and $\alpha f$. We write $\alpha\preccurlyeq\beta$ if $\beta=\alpha\gamma$ for some $\gamma\in\Words$. Two words are \emph{prefix-incomparable} if neither is a prefix of the other.

The infinite path space is $\Cantor$, equipped with the product topology. For $\alpha\in\Words$, the cylinder determined by $\alpha$ is
$$
Z(\alpha)=\{\alpha\xi:\xi\in\Cantor\}.
$$
The sets $Z(\alpha)$ are compact open subsets of $\Cantor$ and form a basis for its topology. We use the cylinder notation of \cite[Section~5]{BrownloweSorensen}.

If $\alpha=a_1\cdots a_r$, put $\alpha^*=a_r^*\cdots a_1^*$. The defining relations of $L$ say that
$$
\alpha^*\beta=
\begin{cases}
1,&\text{if }\alpha=\beta,\\
\gamma,&\text{if }\beta=\alpha\gamma\text{ and }|\gamma|>0,\\
\gamma^*,&\text{if }\alpha=\beta\gamma\text{ and }|\gamma|>0,\\
0,&\text{if }\alpha\text{ and }\beta\text{ are prefix-incomparable}.
\end{cases}
$$
In particular, $\alpha^*\alpha=1$ for every $\alpha\in\Words$.

For later coefficient arguments, we also fix the standard basis obtained by choosing a special edge. Choose $e$ as the special edge at the unique vertex of $R_2$. Let $\mathcal B$ consist of $1$, all nonempty real paths, all nonempty ghost paths, and all monomials $\alpha\beta^*$ with $|\alpha|,|\beta|>0$ for which the last edges of $\alpha$ and $\beta$ are not both $e$. By \cite[Theorem~3.7]{LopatkinNam}, $\mathcal B$ is a basis of $L_A(R_2)$ over every commutative unital coefficient ring $A$. Consequently, if $A$ is a unital subring of $K$, then the canonical homomorphism induced by $A\hookrightarrow K$,
$$
L_A(R_2)\longrightarrow L_K(R_2)
$$
is injective, and its image is the $A$-span of $\mathcal B$.

The terminology for finite subsets of $\Words$ is not uniform. In the Higman--Thompson literature, they occur as finite bases obtained by simple expansion; see \cite[Section~1]{Pardo}. Freeland uses the term \emph{leaf set} in the binary-tree model of Thompson's group $V$ \cite[pp.~41--43]{Freeland}. We follow Freeland, since our matrix construction is based on his definition.
A finite rooted subtree $S$ of $T_2$ is \emph{full} if it contains the root, is closed under predecessors, and every vertex of $S$ has either both children in $S$ or neither child in $S$. A vertex of $S$ with no children in $S$ is a leaf. A finite subset $C\subseteq\Words$ is a \emph{leaf set} if it is the set of leaves of a finite full rooted subtree.

Equivalently, a finite set $C=\{\gamma_1,\ldots,\gamma_n\}$ is a leaf set if
$$
\Cantor=\bigsqcup_{i=1}^n Z(\gamma_i).
$$
A third equivalent description is obtained by simple expansions. A \emph{simple expansion} replaces a word $\gamma$ by its two children $\gamma e$ and $\gamma f$. The leaf sets are precisely the sets obtained from $\{1\}$ by finitely many simple expansions.

An \emph{ordered leaf set} is an ordering $C=(\gamma_1,\ldots,\gamma_n)$ of a leaf set. Until Section~\ref{sec:higher-arity}, leaf sets are understood to belong to the binary tree $T_2$.

\begin{lemma}\label{lem:leaf-set-relations}
A finite set $C=\{\gamma_1,\ldots,\gamma_n\}\subseteq\Words$ is a leaf set if and only if
$$
\gamma_i^*\gamma_j=\delta_{ij}1
\quad\text{for }1\leq i,j\leq n,
\qquad\text{and}\qquad
\sum_{i=1}^n\gamma_i\gamma_i^*=1.
$$
\end{lemma}

\begin{proof}
Suppose first that $C$ is a leaf set. Distinct elements of $C$ are prefix-incomparable, and hence
$
\gamma_i^*\gamma_j=\delta_{ij}1.
$
Every leaf set is obtained from $\{1\}$ by finitely many simple expansions. If a leaf $\gamma$ is replaced by $\gamma e$ and $\gamma f$, then
$$
\gamma e(\gamma e)^*+\gamma f(\gamma f)^*
=\gamma(ee^*+ff^*)\gamma^*
=\gamma\gamma^*.
$$
Thus the sum of the idempotents $\gamma\gamma^*$ is unchanged under simple expansion. It is equal to $1$ for the leaf set $\{1\}$, and therefore $\sum_{i=1}^n\gamma_i\gamma_i^*=1$.

Conversely, suppose that the displayed identities hold. The first family of identities shows that distinct elements of $C$ are prefix-incomparable, so the cylinders $Z(\gamma_i)$ are pairwise disjoint. Suppose that they do not cover $\Cantor$. Choose $\xi\in\Cantor\setminus\bigcup_iZ(\gamma_i)$, and let $\mu$ be a prefix of $\xi$ whose length is greater than every $|\gamma_i|$. Then $\mu$ is prefix-incomparable with every $\gamma_i$, and hence $\gamma_i^*\mu=0$ for every $i$. Multiplying the second displayed identity on the left by $\mu^*$ and on the right by $\mu$ we obtain
$$
1=\mu^*\mu
=\sum_{i=1}^n\mu^*\gamma_i\gamma_i^*\mu
=0,
$$
a contradiction. Hence the cylinders $Z(\gamma_i)$ partition $\Cantor$, and $C$ is a leaf set.
\end{proof}

\begin{lemma}\label{lem:extend-antichain}
Every finite antichain in the prefix order on $\Words$ is contained in a leaf set.
\end{lemma}

\begin{proof}
The assertion is immediate for the empty antichain. Let $F\subseteq\Words$ be a nonempty finite antichain, and choose
$
N>\max\{|\alpha|:\alpha\in F\}.
$
Put
$$
C=F\cup
\{\mu\in\{e,f\}^N:
\text{ no }\alpha\in F\text{ is a prefix of }\mu\}.
$$
The elements adjoined to $F$ have the same length, and none extends an element of $F$. Thus $C$ is an antichain.

Let $\xi\in\Cantor$. If some element of $F$ is a prefix of $\xi$, then $\xi$ belongs to the corresponding cylinder. Otherwise, the prefix of $\xi$ of length $N$ belongs to $C$. Hence the cylinders determined by the elements of $C$ partition $\Cantor$, and $C$ is a leaf set containing $F$.
\end{proof}

A leaf set $D$ is an \emph{expansion} of a leaf set $C$ if $D$ is obtained from $C$ by finitely many simple expansions. We shall also say that $D$ is a \emph{refinement} of $C$. Any two leaf sets have a common refinement. Indeed, if $N$ is at least the length of every word occurring in either leaf set, then both leaf sets expand to $\{e,f\}^N$. This is the usual common-expansion procedure for Higman--Thompson groups \cite[Section~1]{Pardo}.

\begin{lemma}\label{lem:refinement-identity}
For $\mu,\nu\in\Words$ and $r\geq0$, we have
$
\mu\nu^*
=
\sum_{\delta\in\{e,f\}^r}
\mu\delta(\nu\delta)^*.
$
\end{lemma}

\begin{proof}
Iterating the Cuntz--Krieger relation, we have
$
1=\sum_{\delta\in\{e,f\}^r}\delta\delta^*.
$
Multiplying on the left by $\mu$ and on the right by $\nu^*$ proves the asserted identity.
\end{proof}

Let $C=(\gamma_1,\ldots,\gamma_n)$ be an ordered leaf set, and put $\varepsilon_{ij}^{C}=\gamma_i\gamma_j^*$ where $1\leq i,j\leq n$. Lemma~\ref{lem:leaf-set-relations} implies that
$$
\varepsilon_{ij}^{C}\varepsilon_{k\ell}^{C}
=\delta_{jk}\varepsilon_{i\ell}^{C}
\qquad\text{and}\qquad
\sum_{i=1}^n\varepsilon_{ii}^{C}=1.
$$
Thus $\{\varepsilon_{ij}^{C}\}$ is a complete system of matrix units in $L$. The paths $\gamma_1,\ldots,\gamma_n$ determine the following explicit isomorphism $M_n(L)\cong L$; compare \cite[Theorem~4.14]{AbramsAnhPardo}.

\begin{proposition}\label{prop:theta-isomorphism}
For every ordered leaf set $C=(\gamma_1,\ldots,\gamma_n)$, the map
$$
\Theta_C:M_n(L)\longrightarrow L,
\qquad
\Theta_C((a_{ij}))
=\sum_{i,j=1}^n\gamma_i a_{ij}\gamma_j^*,
$$
is a unital $K$-algebra isomorphism. Its inverse is
$
\Theta_C^{-1}(x)
=
(\gamma_i^*x\gamma_j)_{1\leq i,j\leq n}.
$
\end{proposition}

\begin{proof}
Let $X=(x_{ij})$ and $Y=(y_{ij})$ belong to $M_n(L)$. Using $\gamma_j^*\gamma_k=\delta_{jk}1$, we obtain
$$
\Theta_C(X)\Theta_C(Y)
=
\sum_{i,j,k,\ell}
\gamma_i x_{ij}(\gamma_j^*\gamma_k)y_{k\ell}\gamma_\ell^*
=
\sum_{i,j,\ell}
\gamma_i x_{ij}y_{j\ell}\gamma_\ell^*
=\Theta_C(XY).
$$
Moreover,
$
\Theta_C(I_n)
=\sum_{i=1}^n\gamma_i\gamma_i^*
=1.
$
For $1\leq r,s\leq n$, we have
$
\gamma_r^*\Theta_C(X)\gamma_s=x_{rs},
$
so the displayed formula for $\Theta_C^{-1}$ is a left inverse. Conversely, for $x\in L$, we also have
$$
\Theta_C\bigl((\gamma_i^*x\gamma_j)_{i,j}\bigr)
=
\sum_{i,j=1}^n
\gamma_i\gamma_i^*x\gamma_j\gamma_j^*
=
\left(\sum_{i=1}^n\gamma_i\gamma_i^*\right)
x
\left(\sum_{j=1}^n\gamma_j\gamma_j^*\right)
=x.
$$
Hence $\Theta_C$ and the displayed map are mutually inverse.
\end{proof}

\section{Leaf-matrix presentations and elementary transvections}\label{sec:leaf-matrix}

Let $A=(\alpha_1,\ldots,\alpha_n)$ and $B=(\beta_1,\ldots,\beta_n)$ be ordered leaf sets of the same cardinality. For $M=(m_{ij})\in M_n(K)$, put
$$
u(A,M,B)=\sum_{i,j=1}^n m_{ij}\alpha_i\beta_j^*.
$$
We refer to $A$ and $B$ as the range and domain leaf sets, respectively.

\begin{lemma}\label{lem:leaf-matrix-multiplication}
If $A,B,C$ are ordered leaf sets of cardinality $n$ and $M,N\in M_n(K)$, then
$$
u(A,M,B)u(B,N,C)=u(A,MN,C).
$$
\end{lemma}

\begin{proof}
Write $C=(\gamma_1,\ldots,\gamma_n)$. Since $\beta_j^*\beta_k=\delta_{jk}1$,
$$
u(A,M,B)u(B,N,C)
=\sum_{i,j,k,\ell}
m_{ij}n_{k\ell}\alpha_i(\beta_j^*\beta_k)\gamma_\ell^*
=\sum_{i,\ell}(MN)_{i\ell}\alpha_i\gamma_\ell^*
=u(A,MN,C).
$$
\end{proof}

\begin{proposition}\label{prop:leaf-matrix-inverse}
If $M\in\GL_n(K)$, then $u(A,M,B)$ is invertible, with
$$
u(A,M,B)^{-1}=u(B,M^{-1},A).
$$
\end{proposition}

\begin{proof}
Lemma~\ref{lem:leaf-matrix-multiplication} shows that
$u(A,M,B)u(B,M^{-1},A)=u(A,I_n,A)=1$. The reverse product follows by interchanging $A$ and $B$.
\end{proof}

\begin{definition}
A \emph{leaf-matrix presentation} of a unit $u\in L^\times$ is an expression $u=u(A,M,B)$ in which $A$ and $B$ are ordered leaf sets of the same cardinality $n$ and $M\in\GL_n(K)$. Let $\LM_2(K)$ denote the set of units admitting a leaf-matrix presentation, and put $\LGLtwo=\langle\LM_2(K)\rangle\leq L^\times$.
\end{definition}

We next compare this definition with Freeland's Leavitt general linear group. For a fixed ordered leaf set $C$, the elements $u(C,M,C)$, with $M\in\GL_n(K)$, form the copy of $\GL_n(K)$ used in \cite[Definition~4.3.9]{Freeland}. Thus the group generated by Freeland's copies of $\GL_n(K)$ is contained in $\LGLtwo$.

Conversely, every unit in $\LM_2(K)$ has the factorization $u(A,M,B)=u(A,M,A)u(A,I_n,B)$. The first factor belongs to the copy of $\GL_n(K)$ associated with $A$. The second is the element of Thompson's group $V$ represented by the table $\beta_i\xi\mapsto\alpha_i\xi$. This is the standard realization of $V$ inside the unit group of a Leavitt algebra; see \cite[Section~3]{Pardo}.

The group $V$ is generated by transpositions of pairs of prefix-incomparable words \cite[p.~44]{Freeland}. Extend such a pair to an ordered leaf set $C$. The corresponding transposition has the form $u(C,P,C)$, where $P$ is a permutation matrix. Hence $V$ is contained in the group generated by Freeland's copies of $\GL_n(K)$. It follows that $\LGLtwo$ is precisely Freeland's Leavitt general linear group \cite[Definition~4.3.10]{Freeland}. Freeland works over a finite field; the same definition and the preceding identification apply over an arbitrary field.

Under the matrix-ring isomorphisms of Proposition~\ref{prop:theta-isomorphism}, elementary matrices take the following form.

\begin{proposition}\label{prop:scalar-leaf-transvection}
Let $\alpha,\beta\in\Words$ be prefix-incomparable and let $\lambda\in K$. Then $1+\lambda\alpha\beta^*$ belongs to $\LM_2(K)$.
\end{proposition}

\begin{proof}
By Lemma~\ref{lem:extend-antichain}, choose an ordered leaf set $C=(\alpha,\beta,\gamma_3,\ldots,\gamma_n)$. Then
$u(C,I_n+\lambda E_{12},C)=1+\lambda\alpha\beta^*$.
\end{proof}

\begin{lemma}\label{lem:orthogonal-sum}
Let $\alpha_1,\ldots,\alpha_m,\beta_1,\ldots,\beta_m\in\Words$ and $\lambda_1,\ldots,\lambda_m\in K$. If $\beta_s^*\alpha_r=0$ for all $1\leq r,s\leq m$, then $1+\sum_{r=1}^m\lambda_r\alpha_r\beta_r^*$ belongs to $\LGLtwo$.
\end{lemma}

\begin{proof}
Put $t_r=\lambda_r\alpha_r\beta_r^*$. Taking $s=r$ in the hypothesis shows that $\alpha_r$ and $\beta_r$ are prefix-incomparable. Hence $1+t_r\in\LM_2(K)$ by Proposition~\ref{prop:scalar-leaf-transvection}. Moreover, $t_rt_s=0$ for all $r,s$. Therefore
$$
\prod_{r=1}^m(1+t_r)=1+\sum_{r=1}^m t_r\in\LGLtwo.
$$
\end{proof}

\begin{theorem}\label{thm:corner-transvections}
Let $\rho,\sigma\in\Words$ be prefix-incomparable and let $a\in L$. Then $1+\rho a\sigma^*\in\LGLtwo$. Moreover, $(\rho a\sigma^*)^2=0$, and hence $(1+\rho a\sigma^*)^{-1}=1-\rho a\sigma^*$.
\end{theorem}

\begin{proof}
Write $a=\sum_{r=1}^m c_r\mu_r\nu_r^*$, where $c_r\in K$ and $\mu_r,\nu_r\in\Words$. Then
$$
\rho a\sigma^*
=\sum_{r=1}^m c_r(\rho\mu_r)(\sigma\nu_r)^*.
$$
Since $\rho$ and $\sigma$ are prefix-incomparable, the words $\rho\mu_r$ and $\sigma\nu_s$ are prefix-incomparable for every $r,s$. Hence
$(\sigma\nu_s)^*(\rho\mu_r)=0$, and Lemma~\ref{lem:orthogonal-sum} implies that $1+\rho a\sigma^*\in\LGLtwo$. Finally, $\sigma^*\rho=0$, so $(\rho a\sigma^*)^2=0$, which proves the asserted inverse.
\end{proof}

The set $\LM_2(K)$ is not closed under multiplication.

\begin{proposition}\label{prop:single-leaf-not-exhaustive}
The element $x=1+(e+e^2)f^*$ is a unit of $L$ but does not belong to $\LM_2(K)$. Nevertheless,
$x=(1+e^2f^*)(1+ef^*)$, so $x\in\LGLtwo$.
\end{proposition}

\begin{proof}
Put $t=(e+e^2)f^*$. Since $f^*e=f^*e^2=0$, we have $t^2=0$ and hence $x^{-1}=1-t$.

Suppose that $x=u(A,M,B)$, where $A$ and $B$ are ordered leaf sets of cardinality $n$ and $M\in\GL_n(K)$. The infinite word $feee\cdots$ belongs to a unique cylinder $Z(\beta_j)$ with $\beta_j\in B$. If $\beta_j=1$, then $B=\{1\}$ and $n=1$. The only leaf set of cardinality one is $\{1\}$, so $A=\{1\}$ and $x$ is a scalar. This is impossible because $xf=f+e+e^2$ and the real paths are linearly independent.

It follows that $\beta_j=fe^r$ for some $r\geq0$. Multiplying on the right by $\beta_j$, we find
$$
\sum_{i=1}^n m_{ij}\alpha_i
=x\beta_j
=fe^r+e^{r+1}+e^{r+2}.
$$
The real paths are linearly independent. Hence both $e^{r+1}$ and $e^{r+2}$ occur among the elements of $A$, contrary to the fact that a leaf set is an antichain in the prefix order. Therefore $x\notin\LM_2(K)$.

By Proposition~\ref{prop:scalar-leaf-transvection}, both $1+e^2f^*$ and $1+ef^*$ belong to $\LM_2(K)$. Since $f^*e=0$, their product is $x$. Thus $\LM_2(K)$ is not closed under multiplication, while $x\in\LGLtwo$.
\end{proof}

\begin{corollary}\label{cor:elementary-image}
Let $C=(\gamma_1,\ldots,\gamma_n)$ be an ordered leaf set with $n\geq2$. Then $\Theta_C(E_n(L))\leq\LGLtwo$. More precisely, if $i\neq j$ and $a\in L$, then
$\Theta_C(I_n+aE_{ij})=1+\gamma_i a\gamma_j^*$.
\end{corollary}

\begin{proof}
By definition,
$$
\Theta_C(I_n+aE_{ij})=1+\gamma_i a\gamma_j^*.
$$
The leaves $\gamma_i$ and $\gamma_j$ are prefix-incomparable, so this element belongs to $\LGLtwo$ by Theorem~\ref{thm:corner-transvections}. The elementary matrices generate $E_n(L)$, which proves the assertion.
\end{proof}

\section{The full unit group}\label{sec:full-unit-group}

We now determine the group $\LGLtwo$. For a unital ring $R$, let $E_n(R)$ denote the subgroup of $\GL_n(R)$ generated by the elementary matrices $x_{ij}(r)=I_n+rE_{ij}$, where $i\neq j$ and $r\in R$, and let $D_n(R)$ denote the subgroup of invertible diagonal matrices. The ring $R$ is a \emph{$\mathrm{GE}$-ring} if $\GL_n(R)=\langle E_n(R),D_n(R)\rangle$ for every $n\geq2$.

\begin{lemma}\label{lem:GE-reduction}
Let $R$ be a unital $\mathrm{GE}$-ring and let $n\geq2$.
\begin{enumerate}[label=\textup{(\roman*)}]
\item If $c\in[R^\times,R^\times]$, then $\diag(c,1,\ldots,1)\in E_n(R)$.
\item Let $F$ be a central subfield of $R$. If the homomorphism $F^\times\to R^\times_{\mathrm{ab}}$ induced by inclusion is surjective, then $\GL_n(R)=E_n(R)D_n(F)$.
\item If $R^\times$ is perfect, then $\GL_n(R)=E_n(R)$.
\end{enumerate}
\end{lemma}

\begin{proof}
We first prove (i) for $n=2$. For $c\in R^\times$, set
$$
w(c)=x_{12}(c)x_{21}(-c^{-1})x_{12}(c)
=\begin{pmatrix}0&c\\-c^{-1}&0\end{pmatrix}.
$$
Then $w(c)w(-1)=\diag(c,c^{-1})\in E_2(R)$. If $a,b\in R^\times$ and $[a,b]=aba^{-1}b^{-1}$, then
$$
\diag([a,b],1)
=\diag(a,a^{-1})\diag(b,b^{-1})
 \diag(a^{-1}b^{-1},ba)\in E_2(R).
$$
Indeed, each factor on the right has the form $\diag(c,c^{-1})$. Since $[R^\times,R^\times]$ is generated by such commutators, $\diag(c,1)\in E_2(R)$ for every $c\in[R^\times,R^\times]$.

Embedding the preceding calculation in the upper-left $2\times2$ block proves the asserted inclusion in $E_n(R)$. The signed transposition matrix acting on the $i$th and $j$th coordinates is obtained from the same elementary factorization as $w(1)$. Conjugation by this matrix moves the entry $c$ to any prescribed diagonal position. This proves (i).

For (ii), let $D=\diag(u_1,\ldots,u_n)\in D_n(R)$. By surjectivity, for each $i$ there is $\lambda_i\in F^\times$ such that $c_i=u_i\lambda_i^{-1}$ belongs to $[R^\times,R^\times]$. Applying part (i) in each diagonal position, we have $\diag(c_1,\ldots,c_n)\in E_n(R)$. Hence
$D=\diag(c_1,\ldots,c_n)\diag(\lambda_1,\ldots,\lambda_n)\in E_n(R)D_n(F)$. Since $D_n(F)$ normalizes $E_n(R)$, the $\mathrm{GE}$-property now implies $\GL_n(R)=E_n(R)D_n(F)$.

If $R^\times$ is perfect, every diagonal entry of an element of $D_n(R)$ belongs to $[R^\times,R^\times]$. Part (i) therefore implies $D_n(R)\subseteq E_n(R)$, and the $\mathrm{GE}$-property yields $\GL_n(R)=E_n(R)$. This proves (iii).
\end{proof}

We apply this lemma to $L=\Lc$. The graph $R_2$ satisfies the three conditions in \cite[Theorem~11]{AbramsArandaPino}, and hence $L$ is a unital purely infinite simple ring. Ara, Goodearl and Pardo prove that every purely infinite simple ring is a $\mathrm{GE}$-ring and that the natural homomorphism $R^\times\to K_1(R)$ induces an isomorphism $R^\times_{\mathrm{ab}}\cong K_1(R)$; see \cite[Theorem~2.4 and its proof]{AraGoodearlPardo}.

\begin{proposition}\label{prop:Kone-binary}
For $L=\Lc$, we have $K_1(L)=0$. Consequently, $L^\times$ is perfect and $\GL_n(L)=E_n(L)$ for every $n\geq2$.
\end{proposition}

\begin{proof}
A field is regular supercoherent, so the exact sequence of \cite[Theorem~7.6]{AraBrustengaCortinas} applies. The adjacency matrix of $R_2$ is $A_{R_2}=(2)$, and hence $I-A_{R_2}^t=(-1)$. The relevant part of the exact sequence is
$$
K_1(K)\xrightarrow{-1}K_1(K)\longrightarrow K_1(L)
\longrightarrow K_0(K)\xrightarrow{-1}K_0(K).
$$
Under the identifications $K_1(K)=K^\times$ and $K_0(K)=\mathbb Z$, the first map is inversion and the last map is multiplication by $-1$. Both are automorphisms, so exactness shows $K_1(L)=0$.

By \cite[Theorem~2.4]{AraGoodearlPardo}, the canonical map induces $L^\times_{\mathrm{ab}}\cong K_1(L)=0$. Thus $L^\times$ is perfect. Since $L$ is a $\mathrm{GE}$-ring, Lemma~\ref{lem:GE-reduction}(iii) implies $\GL_n(L)=E_n(L)$ for every $n\geq2$.
\end{proof}

\begin{theorem}\label{thm:all-units}
For every field $K$, the Leavitt general linear group is the full unit group:
$$
\LGLtwo=L_K(R_2)^\times.
$$
\end{theorem}

\begin{proof}
Let $C=(e,f)$. Proposition~\ref{prop:theta-isomorphism} induces an isomorphism $\Theta_C:\GL_2(L)\to L^\times$. Proposition~\ref{prop:Kone-binary} and Corollary~\ref{cor:elementary-image} say that
$$
L^\times
=\Theta_C\bigl(\GL_2(L)\bigr)
=\Theta_C\bigl(E_2(L)\bigr)
\leq\LGLtwo.
$$
The reverse inclusion follows from the definition of $\LGLtwo$ as a subgroup of $L^\times$.
\end{proof}

Together, Proposition~\ref{prop:single-leaf-not-exhaustive} and Theorem~\ref{thm:all-units} show that $\LM_2(K)\subsetneq L^\times$, whereas $\langle\LM_2(K)\rangle=L^\times$. Thus a unit need not admit a leaf-matrix presentation, but every unit is a finite product of units admitting such a presentation.

\begin{corollary}\label{cor:fixed-corner-generation}
Let $C=(\gamma_1,\ldots,\gamma_n)$ be an ordered leaf set with $n\geq2$. Then
$$
L^\times
=\Theta_C\bigl(E_n(L)\bigr)
=\left\langle 1+\gamma_i a\gamma_j^*:a\in L,\ i\neq j\right\rangle.
$$
In particular,
$$
L^\times
=\left\langle 1+eaf^*,\ 1+fbe^*:a,b\in L\right\rangle.
$$
\end{corollary}

\begin{proof}
Proposition~\ref{prop:theta-isomorphism} identifies $L^\times$ with $\Theta_C(\GL_n(L))$, while Proposition~\ref{prop:Kone-binary} identifies $\GL_n(L)$ with $E_n(L)$. This proves the first equality. The second follows from Corollary~\ref{cor:elementary-image}, since the matrices $I_n+aE_{ij}$ generate $E_n(L)$. Taking $C=(e,f)$ proves the final formula.
\end{proof}

\section{Finite generation}\label{sec:finite-generation}

The finite-generation criterion uses the following coefficient argument. We identify $K$ with the scalar subfield $K\cdot1\subseteq L_K(R_2)$, and likewise identify every unital subring $A\subseteq K$ with $A\cdot1$.

\begin{lemma}\label{lem:coefficient-descent}
Let $u_1,\ldots,u_m\in L^\times$, and let $A$ be the unital subring of $K$ generated by the coefficients occurring in the unique $\mathcal B$-expansions of $u_1,u_1^{-1},\ldots,u_m,u_m^{-1}$. Then $\langle u_1,\ldots,u_m\rangle\subseteq L_A(R_2)^\times$. Moreover,
$$
L_A(R_2)\cap K=A
$$
inside $L_K(R_2)$.
\end{lemma}

\begin{proof}
By the choice of $A$, the elements $u_i$ and $u_i^{-1}$ belong to $L_A(R_2)$. Hence $u_i\in L_A(R_2)^\times$, and every word in the $u_i^{\pm1}$ belongs to $L_A(R_2)^\times$.

The basis described above identifies $L_A(R_2)$ with the $A$-span of $\mathcal B$ and $L_K(R_2)$ with the $K$-span of the same set. Since $1\in\mathcal B$, uniqueness of the $\mathcal B$-expansion forces $L_A(R_2)\cap K=A$.
\end{proof}

We shall also use the following standard consequence of Zariski's lemma.

\begin{lemma}\label{lem:field-fg-ring}
A field which is finitely generated as a ring is finite.
\end{lemma}

\begin{proof}
Let $F$ be a field which is finitely generated as a ring. Suppose first that $\operatorname{char}F=p>0$. Then $F$ is a finitely generated $\mathbb F_p$-algebra. Zariski's lemma shows that $F$ is a finite algebraic extension of $\mathbb F_p$, and hence $F$ is finite.

Suppose that $\operatorname{char}F=0$, and write $F=\mathbb Z[\xi_1,\ldots,\xi_r]$. Since $\mathbb Q\subseteq F$, we have $F=\mathbb Q[\xi_1,\ldots,\xi_r]$. Zariski's lemma therefore implies $[F:\mathbb Q]<\infty$. Choose a nonzero integer $N$ such that every $\xi_i$ is integral over $\mathbb Z[1/N]$. Since $1/N\in F$, we have
$F=\mathbb Z[1/N][\xi_1,\ldots,\xi_r]$, and consequently $F$ is integral over $\mathbb Z[1/N]$.

If a field is integral over a subring, then the subring is a field. Indeed, if $0\neq b$ belongs to the subring, then $b^{-1}$ is integral over it, and an integral equation for $b^{-1}$ places $b^{-1}$ in the subring. This would make $\mathbb Z[1/N]$ a field, which is impossible. Thus the characteristic-zero case cannot occur.
\end{proof}

\begin{theorem}\label{thm:finite-generation}
Let $K$ be a field. Then $L_K(R_2)^\times$ is finitely generated if and only if $K$ is finite.
\end{theorem}

\begin{proof}
Suppose first that $K$ is finite. Then $L$ is finitely generated as a unital ring; for example, it is generated by the finite set $K\cup\{e,f,e^*,f^*\}$. Choose a finite set $S\subseteq L$ which generates $L$ as a unital ring, and let $C=(\gamma_1,\gamma_2,\gamma_3)=(ee,ef,f)$.

Let $H$ be the subgroup of $E_3(L)$ generated by the finite set
$$
\{x_{ij}(s):i\neq j,\ s\in S\cup\{1\}\}.
$$
Put $T=\{a\in L:x_{ij}(a)\in H\text{ for every }i\neq j\}$. The identities
$$
x_{ij}(a+b)=x_{ij}(a)x_{ij}(b)
\qquad
\text{and}
\qquad
x_{ij}(a)^{-1}=x_{ij}(-a)
$$
show that $T$ is closed under addition and additive inverses. If $a,b\in T$ and $i\neq j$, choose $k$ distinct from $i$ and $j$. The elementary commutator identity
$$
[x_{ik}(a),x_{kj}(b)]=x_{ij}(ab)
$$
shows that $ab\in T$. Since $1\in T$, the set $T$ is a unital subring of $L$. It contains $S$, and therefore $T=L$.

It follows that $H$ contains every elementary matrix $x_{ij}(a)$ with $a\in L$. Hence $H=E_3(L)$, so $E_3(L)$ is finitely generated. By Corollary~\ref{cor:fixed-corner-generation},
$L^\times=\Theta_C(E_3(L))$. Since
$\Theta_C(x_{ij}(s))=1+\gamma_i s\gamma_j^*$, the group $L^\times$ is generated by the finite set
$$
\{1+\gamma_i s\gamma_j^*:i\neq j,\ s\in S\cup\{1\}\}.
$$

Conversely, suppose that $L^\times=\langle u_1,\ldots,u_m\rangle$, and let $A$ be the finitely generated unital subring of $K$ obtained from Lemma~\ref{lem:coefficient-descent}. For every $\lambda\in K^\times$, the scalar unit $\lambda1$ is a word in the $u_i^{\pm1}$. Hence $\lambda\in L_A(R_2)\cap K=A$. Thus $K^\times\subseteq A$, and therefore $K=A$. It follows that $K$ is finitely generated as a ring. Lemma~\ref{lem:field-fg-ring} now shows that $K$ is finite.
\end{proof}

\section{The monomial subgroup and $\GL_\infty(K)$}
\label{sec:subgroups}

Recall that a matrix is monomial if each row and each column contains exactly one nonzero entry. Let $\Mtwo$ be the set of units admitting a leaf-matrix presentation $u(A,M,B)$ with $M$ monomial. After reordering $A$, such a presentation has the form
$$
w=\sum_{i=1}^n\lambda_i\alpha_i\beta_i^*
\qquad \text{where }\lambda_i\in K^\times.
$$
The elements for which all $\lambda_i$ are equal to $1$ form the usual copy of the Higman--Thompson group $V=G_{2,1}$ in $L^\times$; see \cite[Sections~1 and~3]{Pardo}.

Let $\mathcal D$ denote the diagonal subalgebra of $L$, that is,
$$
\mathcal D=\operatorname{span}_K\{\alpha\alpha^*:\alpha\in\Words\}.
$$
The assignment $\alpha\alpha^*\mapsto 1_{Z(\alpha)}$ identifies $\mathcal D$ with $C_{\mathrm{lc}}(\Cantor,K)$, the algebra of locally constant $K$-valued functions on $\Cantor$. Consequently,
$$
\mathcal D^\times\cong C_{\mathrm{lc}}(\Cantor,K^\times),
$$
where $K^\times$ carries the discrete topology.

\begin{proposition}\label{prop:monomial-subgroup}
The set $\Mtwo$ is a subgroup of $L^\times$, and
$$
\Mtwo\cong C_{\mathrm{lc}}(\Cantor,K^\times)\rtimes V.
$$
Here $V$ acts by pullback: $(g\cdot\mu)(x)=\mu(g^{-1}x)$. Thus
$$
(\lambda,g)(\mu,h)
=\bigl(\lambda\,(g\cdot\mu),gh\bigr),
$$
where $gh=g\circ h$.
\end{proposition}

\begin{proof}
Choose a leaf-matrix presentation $w=\sum_i\lambda_i\alpha_i\beta_i^*$ whose coefficient matrix is monomial. Put
$$
d_\lambda=\sum_i\lambda_i\alpha_i\alpha_i^*
\qquad\text{and}\qquad
U_g=\sum_i\alpha_i\beta_i^*,
$$
where $g\in V$ is represented by the table $g(\beta_i\xi)=\alpha_i\xi$. Then $d_\lambda\in\mathcal D^\times$, $U_g\in V$, and $w=d_\lambda U_g$.

Conversely, let $d\in\mathcal D^\times$ and $g\in V$. Choose a table $g(\beta_i\xi)=\alpha_i\xi$. By refining this table, we may assume that the locally constant function corresponding to $d$ is constant on each cylinder $Z(\alpha_i)$. If its value there is $\lambda_i$, then
$$
dU_g=\sum_i\lambda_i\alpha_i\beta_i^*\in\Mtwo.
$$
It follows that $\Mtwo=\mathcal D^\times V$.

For $\mu\in C_{\mathrm{lc}}(\Cantor,K^\times)$, conjugation by $U_g$ satisfies
$$
U_gd_\mu U_g^{-1}=d_{g\cdot\mu}
\qquad
\text{and}
\qquad
(g\cdot\mu)(x)=\mu(g^{-1}x).
$$
Thus $V$ normalizes $\mathcal D^\times$. If $U_g\in\mathcal D^\times$, then conjugation by $U_g$ is trivial on $\mathcal D$. Hence $\mu\circ g^{-1}=\mu$ for every locally constant function $\mu$. Since characteristic functions of cylinders separate points of $\Cantor$, we have $g=1$, and therefore $U_g=1$. Thus $\mathcal D^\times\cap V=\{1\}$.

We have therefore obtained the internal semidirect product $\Mtwo=\mathcal D^\times\rtimes V$. The stated multiplication follows from
$$
d_\lambda U_gd_\mu U_h
=d_\lambda\bigl(U_gd_\mu U_g^{-1}\bigr)U_{gh}
=d_{\lambda\,(g\cdot\mu)}U_{gh}.
$$
\end{proof}

Under this identification, $V$ corresponds to $\{1\}\rtimes V$, and hence
$$
V\leq\Mtwo\leq L^\times.
$$
If $K^\times\neq\{1\}$, choose $\lambda\neq1$. Then $\lambda ee^*+ff^*$ belongs to $\mathcal D^\times\setminus\{1\}$ and hence to $\Mtwo\setminus V$. If $K=\mathbb F_2$, then $\mathcal D^\times=\{1\}$ and $\Mtwo=V$.

The inclusion $\Mtwo\leq L^\times$ is always proper. Indeed, put
$$
q=ee^*+ef^*+ff^*
=u\left((e,f),
\begin{pmatrix}1&1\\0&1\end{pmatrix},
(e,f)\right).
$$
The coefficient matrix is invertible, so $q\in L^\times$. Suppose that $q\in\Mtwo$. By refining a leaf-matrix presentation with monomial coefficient matrix, we may assume that its domain leaf set is $\{e,f\}^N$ for some $N\geq1$. It would then follow that $q\mu$ is a nonzero scalar multiple of a single real path for every $\mu\in\{e,f\}^N$. On the other hand,
$$
qf^N=ef^{N-1}+f^N.
$$
The two terms on the right are distinct real paths, and their sum cannot be a scalar multiple of one real path. Therefore $q\notin\Mtwo$, and hence $\Mtwo\subsetneq L^\times$.

We next embed $\GL_\infty(K)$ in $L^\times$. For $i\geq1$, put $\alpha_i=e^{i-1}f$ and $F_{ij}=\alpha_i\alpha_j^*$. The words $\alpha_i$ are pairwise prefix-incomparable, and hence
$$
F_{ij}F_{k\ell}=\delta_{jk}F_{i\ell}.
$$
For $n\geq1$, put $p_n=\sum_{i=1}^nF_{ii}$.

\begin{proposition}\label{prop:GL-infinity}
There is an injective homomorphism
$$
\Phi:\GL_\infty(K)=\varinjlim_n\GL_n(K)\longrightarrow L^\times,
$$
where the transition maps are $M\mapsto\operatorname{diag}(M,1)$. Moreover, $\Phi(\GL_\infty(K))\subseteq\LM_2(K)$.
\end{proposition}

\begin{proof}
For $M=(m_{ij})\in\GL_n(K)$, define
$$
\Phi_n(M)=1-p_n+\sum_{i,j=1}^nm_{ij}F_{ij}.
$$
Let
$$
C_n=(f,ef,\ldots,e^{n-1}f,e^n).
$$
This is an ordered leaf set. Repeated use of $1=ee^*+ff^*$ shows that
$$
1-p_n=e^n(e^n)^*,
$$
and consequently
$$
\Phi_n(M)
=u\left(C_n,\operatorname{diag}(M,1),C_n\right).
$$
Thus $\Phi_n(M)\in\LM_2(K)$, with inverse $\Phi_n(M^{-1})$.

The matrix-unit relations, together with $(1-p_n)F_{ij}=F_{ij}(1-p_n)=0$ for $1\leq i,j\leq n$, imply $\Phi_n(MN)=\Phi_n(M)\Phi_n(N)$. Moreover, $p_{n+1}=p_n+F_{n+1,n+1}$, and therefore
$$
\Phi_{n+1}\bigl(\operatorname{diag}(M,1)\bigr)=\Phi_n(M).
$$
The maps $\Phi_n$ consequently induce a homomorphism $\Phi:\GL_\infty(K)\to L^\times$.

Finally, for $1\leq r,s\leq n$, we have
$$
\alpha_r^*\Phi_n(M)\alpha_s=m_{rs}1.
$$
Hence every $\Phi_n$ is injective, and so is the induced homomorphism $\Phi$. The formula for $\Phi_n(M)$ also shows that the image of $\Phi$ is contained in $\LM_2(K)$.
\end{proof}

\section{Roses with $d$ loops}\label{sec:higher-arity}

Let $d\geq2$, let $R_d$ be the rose with one vertex and loops $a_1,\ldots,a_d$, and put $L_d=L_K(R_d)\cong L_K(1,d)$. Thus
$$
a_i^*a_j=\delta_{ij}1
\qquad
\text{and}
\qquad
\sum_{i=1}^d a_ia_i^*=1.
$$
We identify the finite paths in $R_d$ with the free monoid $\{a_1,\ldots,a_d\}^*$ and regard them as the vertices of the rooted $d$-ary tree. The language of bases and simple expansions for this tree is standard in the theory of Higman--Thompson groups; see \cite[Section~1]{Pardo}.

A finite rooted subtree is full if it contains the root, is closed under predecessors, and every internal vertex has all $d$ children. A \emph{$d$-ary leaf set} is the set of leaves of such a subtree. Equivalently, it is obtained from $\{1\}$ by finitely many simple expansions
$$
\gamma\longmapsto
\{\gamma a_1,\ldots,\gamma a_d\}.
$$
An ordered $d$-ary leaf set is an ordering of a $d$-ary leaf set.

If $A=(\alpha_1,\ldots,\alpha_n)$ and $B=(\beta_1,\ldots,\beta_n)$ are ordered $d$-ary leaf sets and $M=(m_{ij})\in\GL_n(K)$, put
$$
u(A,M,B)=\sum_{i,j=1}^n m_{ij}\alpha_i\beta_j^*.
$$
The multiplication and inverse formulas from Lemma~\ref{lem:leaf-matrix-multiplication} and Proposition~\ref{prop:leaf-matrix-inverse} remain valid. Let $\LM_d(K)$ denote the set of units admitting a presentation $u(A,M,B)$ with $M\in\GL_n(K)$, and put $\LGL_d(K)=\langle\LM_d(K)\rangle\leq L_d^\times$.

If $C=(\gamma_1,\ldots,\gamma_n)$ is an ordered $d$-ary leaf set, then
$$
\gamma_i^*\gamma_j=\delta_{ij}1
\qquad
\text{and}
\qquad
\sum_{i=1}^n\gamma_i\gamma_i^*=1.
$$
Consequently, the map
$$
\Theta_C:M_n(L_d)\longrightarrow L_d,
\qquad
\text{by}
\qquad
\Theta_C((x_{ij}))
=\sum_{i,j=1}^n\gamma_i x_{ij}\gamma_j^*
$$
defines a unital $K$-algebra isomorphism with inverse $x\mapsto(\gamma_i^*x\gamma_j)_{i,j}$, as in Proposition~\ref{prop:theta-isomorphism}.

\begin{proposition}\label{prop:d-ary-elementary}
Let $C=(\gamma_1,\ldots,\gamma_n)$ be an ordered $d$-ary leaf set with $n\geq2$. Then
\begin{equation}\label{eq:d-ary-elementary}
\Theta_C\bigl(E_n(L_d)\bigr)\leq\LGL_d(K).
\end{equation}
More precisely, if $i\neq j$ and $a\in L_d$, then
$$
1+\gamma_i a\gamma_j^*\in\LGL_d(K).
$$
\end{proposition}

\begin{proof}
Every finite antichain in the rooted $d$-ary tree is contained in a $d$-ary leaf set. Indeed, if $F$ is such an antichain, choose an integer $N$ greater than the length of every word in $F$ and adjoin all words $\mu$ of length $N$ for which no element of $F$ is a prefix. The cylinders determined by the resulting set partition the infinite path space.

It follows that if $\alpha$ and $\beta$ are prefix-incomparable and $\lambda\in K$, then $1+\lambda\alpha\beta^*$ belongs to $\LM_d(K)$: extend $\{\alpha,\beta\}$ to an ordered $d$-ary leaf set $D$ and use the elementary matrix $I+\lambda E_{12}$ in the copy of $M_{|D|}(K)$ associated with $D$.

Now write $a=\sum_{r=1}^m c_r\mu_r\nu_r^*$, where $c_r\in K$ and $\mu_r,\nu_r$ are finite paths. Then
$$
\gamma_i a\gamma_j^*
=\sum_{r=1}^m
c_r(\gamma_i\mu_r)(\gamma_j\nu_r)^*.
$$
Since $\gamma_i$ and $\gamma_j$ are prefix-incomparable, $\gamma_i\mu_r$ and $\gamma_j\nu_s$ are prefix-incomparable for every $r,s$. Hence all products between the summands on the right are zero, and
$$
1+\gamma_i a\gamma_j^*
=
\prod_{r=1}^m
\left(1+c_r(\gamma_i\mu_r)(\gamma_j\nu_r)^*\right)
\in\LGL_d(K).
$$
Finally, $\Theta_C(I_n+aE_{ij})=1+\gamma_i a\gamma_j^*$. Since the elementary matrices generate $E_n(L_d)$, \eqref{eq:d-ary-elementary} follows.
\end{proof}

\begin{theorem}\label{thm:d-ary-all-units}
For every field $K$ and every $d\geq2$, we have
$$
\LGL_d(K)=L_K(R_d)^\times.
$$
\end{theorem}

\begin{proof}
The graph $R_d$ satisfies the conditions of \cite[Theorem~11]{AbramsArandaPino}, so $L_d$ is a unital purely infinite simple ring. By \cite[Theorem~2.4 and its proof]{AraGoodearlPardo}, $L_d$ is a $\mathrm{GE}$-ring and the canonical map induces an isomorphism
$(L_d^\times)_{\mathrm{ab}}\cong K_1(L_d)$.

Since a field is regular supercoherent, the exact sequence of \cite[Theorem~7.6]{AraBrustengaCortinas} applies:
$$
K^\times\xrightarrow{\lambda\mapsto\lambda^{\,1-d}}K^\times
\longrightarrow K_1(L_d)
\longrightarrow\mathbb Z\xrightarrow{\,1-d\,}\mathbb Z.
$$
Multiplication by $1-d$ on $\mathbb Z$ is injective. Hence
$$
K_1(L_d)
\cong
\operatorname{coker}
\big(K^\times\xrightarrow{\lambda\mapsto\lambda^{\,1-d}}K^\times\big)
\cong
K^\times/(K^\times)^{d-1}.
$$
In particular, the homomorphism $K^\times\to K_1(L_d)\cong(L_d^\times)_{\mathrm{ab}}$ induced by the scalar inclusion is surjective. Lemma~\ref{lem:GE-reduction}(ii), applied with the central subfield $K\subseteq L_d$, therefore implies
$\GL_d(L_d)=E_d(L_d)D_d(K)$.

Take $C=(a_1,\ldots,a_d)$. The isomorphism $\Theta_C:M_d(L_d)\to L_d$ sends $E_d(L_d)$ into $\LGL_d(K)$ by Proposition~\ref{prop:d-ary-elementary}. It also sends every $D\in D_d(K)$ to $u(C,D,C)\in\LM_d(K)$. Consequently,
$$
L_d^\times
=\Theta_C\bigl(\GL_d(L_d)\bigr)
\leq\LGL_d(K).
$$
The reverse inclusion holds by definition.
\end{proof}

\begin{corollary}\label{cor:d-ary-finite-generation}
For every $d\geq2$, the group $L_K(R_d)^\times$ is finitely generated if and only if $K$ is finite.
\end{corollary}

\begin{proof}
Suppose first that $L_d^\times$ is generated by $u_1,\ldots,u_m$. Choose $a_1$ as the special edge, and let $\mathcal B_d$ be the corresponding basis of $L_d$. Let $A$ be the unital subring of $K$ generated by the coefficients occurring in the $\mathcal B_d$-expansions of $u_1^{\pm1},\ldots,u_m^{\pm1}$. The proof of Lemma~\ref{lem:coefficient-descent} applies verbatim and shows that
$$
\langle u_1,\ldots,u_m\rangle
\subseteq L_A(R_d)^\times
\qquad
\text{and}
\qquad
L_A(R_d)\cap K=A.
$$
Every scalar $\lambda\in K^\times$ is a unit of $L_d$, and hence belongs to $L_A(R_d)$. It follows that $\lambda\in A$ for every $\lambda\in K^\times$. Therefore $A=K$, so $K$ is finitely generated as a ring. Lemma~\ref{lem:field-fg-ring} shows that $K$ is finite.

Conversely, suppose that $K$ is finite. Then $L_d$ is finitely generated as a unital ring. Choose a finite set $S\subseteq L_d$ which generates $L_d$ as a unital ring. Expand the leaf $a_1$ in $\{a_1,\ldots,a_d\}$, and order the resulting leaf set as
$$
C=(a_1a_1,\ldots,a_1a_d,a_2,\ldots,a_d).
$$
Put $n=|C|=2d-1$, so $n\geq3$. Let $H$ be the subgroup of $E_n(L_d)$ generated by the finite set
$$
\{I_n+sE_{ij}:i\neq j,\ s\in S\cup\{1\}\}.
$$
Set $T=\{a\in L_d:I_n+aE_{ij}\in H\text{ for every }i\neq j\}$. The elementary identities
$x_{ij}(a+b)=x_{ij}(a)x_{ij}(b)$, $x_{ij}(a)^{-1}=x_{ij}(-a)$, and
$[x_{ik}(a),x_{kj}(b)]=x_{ij}(ab)$ for distinct $i,j,k$ show that $T$ is a unital subring of $L_d$. Since $S\subseteq T$, we have $T=L_d$, and hence $H=E_n(L_d)$. Thus $E_n(L_d)$ is finitely generated.

The scalar map $K^\times\to(L_d^\times)_{\mathrm{ab}}$ is surjective by the proof of Theorem~\ref{thm:d-ary-all-units}. Lemma~\ref{lem:GE-reduction}(ii) therefore implies
$\GL_n(L_d)=E_n(L_d)D_n(K)$. The group $D_n(K)$ is finite, so $\GL_n(L_d)$ is finitely generated. Finally, $\Theta_C$ induces an isomorphism $\GL_n(L_d)\cong L_d^\times$. Hence $L_d^\times$ is finitely generated.
\end{proof}

\section{Finite presentability and unstable $K_2$}
\label{sec:finite-presentability}

The preceding results identify the full unit groups and determine when they are finitely generated. Since every finitely presented group is finitely generated, Corollary~\ref{cor:d-ary-finite-generation} reduces the question of finite presentability to finite coefficient fields. Throughout this section, $K=\mathbb F_q$ and $L_d=L_K(R_d)$. The finiteness of $K$ will also ensure that $L_d$ is finitely presented as a unital ring and that the diagonal groups $D_n(K)$ are finite.

We first determine the matrix sizes arising from $d$-ary leaf sets. A simple expansion replaces one leaf by its $d$ children and therefore increases the number of leaves by $d-1$. Starting from $\{1\}$, $r$ successive simple expansions produce a leaf set of cardinality $1+r(d-1)$, and every finite $d$-ary leaf set is obtained in this way. Fix $r\geq0$, put $n=1+r(d-1)$, and assume that $n\geq5$. Choose an ordered $d$-ary leaf set $C$ of cardinality $n$. The map $\Theta_C$ restricts to an isomorphism $\GL_n(L_d)\to L_d^\times$. The lower bound on $n$ is needed for the universal central extension used below.

We recall the required unstable $K$-theory. Let $R$ be a unital ring and $n\geq3$. For $i\neq j$ and $a\in R$, write $x_{ij}(a)=I_n+aE_{ij}$, and let $E_n(R)$ be the subgroup of $\GL_n(R)$ generated by these elementary matrices. The Steinberg group $\St_n(R)$ has generators $X_{ij}(a)$, indexed by $i\neq j$ and $a\in R$, subject to the Steinberg relations
\[
\begin{aligned}
X_{ij}(a)X_{ij}(b)&=X_{ij}(a+b),\\
[X_{ij}(a),X_{k\ell}(b)]&=1
   &&(i\neq\ell,\ j\neq k),\\
[X_{ij}(a),X_{jk}(b)]&=X_{ik}(ab)
   &&(i,j,k\text{ distinct}).
\end{aligned}
\]
The elementary matrices satisfy the same relations. Thus $X_{ij}(a)\mapsto x_{ij}(a)$ defines a surjection
$$
\phi_n:\St_n(R)\longrightarrow E_n(R).
$$
Its kernel $K_2(n,R)=\ker\phi_n$ is the unstable $K_2$-group in rank $n$. For $n\geq3$, the group $E_n(R)$ is perfect: if $i,j,k$ are distinct, then $x_{ij}(a)=[x_{ik}(a),x_{kj}(1)]$.

The direct limits $E(R)=\varinjlim E_n(R)$ and $\St(R)=\varinjlim\St_n(R)$ are the stable elementary and Steinberg groups. The stable $K_2$-group is $K_2(R)=\ker(\St(R)\to E(R))$, and stabilization induces a natural homomorphism $\kappa_n:K_2(n,R)\to K_2(R)$. No stability assertion about $\kappa_n$ will be used.

The passage from the Steinberg group to the elementary group rests on the following standard fact about central extensions.

\begin{lemma}\label{lem:central-finite-presentation}
Let $1\to A\to\widetilde G\to G\to1$ be a central extension. If $\widetilde G$ is finitely presented, then $G$ is finitely presented if and only if $A$ is finitely generated.
\end{lemma}

\begin{proof}
Suppose that $A=\langle a_1,\ldots,a_s\rangle$. Adjoining the relations $a_1=\cdots=a_s=1$ to a finite presentation of $\widetilde G$ produces a finite presentation of $G$.

Conversely, write $\widetilde G=F/N$, where $F$ is free of finite rank and $N$ is finitely normally generated. Let $M$ be the kernel of the composite map $F\to\widetilde G\to G$. Then $N\leq M$, $G\cong F/M$, and $A\cong M/N$. Since $G$ is finitely presented and $F$ has finite rank, $M$ is finitely normally generated in $F$. The images of a finite set of normal generators of $M$ normally generate $M/N$ in $F/N$. Since $A$ is central in $\widetilde G$, these images generate $A$ as a group.
\end{proof}

The Steinberg extension relates finite presentability of the unit group to finite generation of $K_2(n,L_d)$.

\begin{theorem}\label{thm:finite-presentation-reduction}
Let $K=\mathbb F_q$, let $d\geq2$, and let $n=1+r(d-1)\geq5$ for some $r\geq0$. The following conditions are equivalent:
\begin{enumerate}[label=\textup{(\roman*)}]
\item $L_d^\times$ is finitely presented;
\item $L_d^\times/K^\times$ is finitely presented;
\item $E_n(L_d)$ is finitely presented;
\item $K_2(n,L_d)$ is finitely generated.
\end{enumerate}
Moreover, $K_2(n,L_d)\cong H_2(E_n(L_d),\mathbb Z)$. Consequently, if these conditions hold for one such $n$, then they hold for every such $n$.
\end{theorem}

\begin{proof}
We first verify that $\St_n(L_d)$ is finitely presented. Write $q=p^m$, choose an irreducible polynomial $\bar f\in\mathbb F_p[t]$ of degree $m$, and let $f\in\mathbb Z[t]$ be a lift of $\bar f$. Then $K\cong\mathbb Z[t]/(p,f(t))$. The algebra $L_d$ has the finite $K$-algebra presentation
$$
L_d\cong K\langle a_1,\ldots,a_d,b_1,\ldots,b_d\rangle
\big/\big\langle b_i a_j-\delta_{ij}1,\ \sum_{i=1}^d a_i b_i-1\big\rangle.
$$
Combining these presentations and adjoining relations which make $t$ commute with the $a_i$ and $b_i$ shows that $L_d$ is finitely presented as a unital ring. By \cite[Theorem~3]{KrsticMcCool}, $\St_n(L_d)$ is finitely presented for every $n\geq4$.

We next compare the elementary group with the unit group. The map $\Theta_C$ restricts to an isomorphism $\GL_n(L_d)\to L_d^\times$. The proof of Theorem~\ref{thm:d-ary-all-units}, applied in rank $n$, establishes $\GL_n(L_d)=E_n(L_d)D_n(K)$. The finite group $D_n(K)$ normalizes $E_n(L_d)$, so $E_n(L_d)$ has finite index in $\GL_n(L_d)$. Since finite presentability is invariant under passage between a group and a finite-index subgroup, $L_d^\times$ is finitely presented if and only if $E_n(L_d)$ is finitely presented. Moreover, $K^\times$ is a finite central subgroup of $L_d^\times$, and a group is finitely presented if and only if its quotient by a finite normal subgroup is finitely presented. This proves the equivalence of \textup{(i)}, \textup{(ii)}, and \textup{(iii)}.

For $n\geq5$, the Steinberg extension
\begin{equation}\label{eq:unstable-Steinberg-extension}
1\longrightarrow K_2(n,L_d)\longrightarrow\St_n(L_d)
\xrightarrow{\phi_n}E_n(L_d)\longrightarrow1
\end{equation}
is the universal central extension of $E_n(L_d)$; see \cite[Chapter~III, Proposition~5.5.1]{Weibel}. The kernel of the universal central extension of a perfect group is its Schur multiplier; see \cite[Chapter~III, Recognition Theorem~5.4]{Weibel}. Hence
\begin{equation}\label{eq:unstable-Schur-multiplier}
K_2(n,L_d)\cong H_2(E_n(L_d),\mathbb Z).
\end{equation}
Since $\St_n(L_d)$ is finitely presented, Lemma~\ref{lem:central-finite-presentation} applied to \eqref{eq:unstable-Steinberg-extension} shows that $E_n(L_d)$ is finitely presented if and only if $K_2(n,L_d)$ is finitely generated. All four conditions are therefore equivalent. Condition~\textup{(i)} is independent of $n$, so the equivalence holds simultaneously for every $n=1+r(d-1)\geq5$.
\end{proof}

The theorem involves unstable $K$-theory. The stable $K_2$-group can be computed explicitly, but this calculation does not determine $K_2(n,L_d)$.

\begin{proposition}\label{prop:stable-Ktwo-Leavitt}
For every finite field $K=\mathbb F_q$ and every $d\geq2$,
$$
K_2(L_d)\cong\{\lambda\in K^\times:\lambda^{d-1}=1\}.
$$
Consequently, $K_2(L_d)$ is cyclic of order $\gcd(d-1,q-1)$. In particular, $K_2(L_d)=0$ if $K=\mathbb F_2$ or $d=2$.
\end{proposition}

\begin{proof}
The adjacency matrix of $R_d$ is $A_{R_d}=(d)$, so $I-A_{R_d}^t=(1-d)$. Since a field is regular supercoherent, the graph $K$-theory exact sequence of \cite[Theorem~7.6]{AraBrustengaCortinas} applies. Using $K_1(K)=K^\times$, its relevant terms are
$$
K_2(K)\xrightarrow{\,1-d\,}K_2(K)\longrightarrow K_2(L_d)
\longrightarrow K^\times\xrightarrow{\lambda\mapsto\lambda^{\,1-d}}K^\times.
$$
The calculation $K_2(\mathbb F_q)=0$ from \cite[Chapter~III, Corollary~6.1.1]{Weibel} reduces exactness to
$$
K_2(L_d)\cong
\ker\bigl(K^\times\xrightarrow{\lambda\mapsto\lambda^{\,1-d}}K^\times\bigr)
=\{\lambda\in K^\times:\lambda^{d-1}=1\}.
$$
Since $K^\times$ is cyclic of order $q-1$, this subgroup is cyclic of order $\gcd(d-1,q-1)$.
\end{proof}

The finiteness of $K_2(L_d)$ reduces the question to the kernel of the stabilization homomorphism.

\begin{corollary}\label{cor:stabilization-obstruction}
With the notation of Theorem~\ref{thm:finite-presentation-reduction}, the group $L_d^\times$ is finitely presented if and only if $\ker\kappa_n$ is finitely generated. In particular, if $\kappa_n:K_2(n,L_d)\to K_2(L_d)$ is injective for some $n=1+r(d-1)\geq5$, where $r\geq0$, then both $L_d^\times$ and $L_d^\times/K^\times$ are finitely presented.
\end{corollary}

\begin{proof}
Proposition~\ref{prop:stable-Ktwo-Leavitt} shows that $K_2(L_d)$ is finite. The quotient
$$
K_2(n,L_d)/\ker\kappa_n\cong\operatorname{im}\kappa_n
$$
is therefore finite. Thus $K_2(n,L_d)$ is finitely generated if and only if $\ker\kappa_n$ is finitely generated. Theorem~\ref{thm:finite-presentation-reduction} completes the proof.
\end{proof}

When $K=\mathbb F_2$, we have $D_n(K)=\{I_n\}$, and $\Theta_C$ identifies $E_n(L_d)$ with $L_d^\times$.

\begin{corollary}\label{cor:Ftwo-finite-presentation}
Let $K=\mathbb F_2$ and let $n=1+r(d-1)\geq5$, where $r\geq0$. Then
$$
K_2(n,L_d)\cong H_2(L_d^\times,\mathbb Z),
$$
and $L_d^\times$ is finitely presented if and only if $H_2(L_d^\times,\mathbb Z)$ is finitely generated. If $\kappa_n$ is injective, then $K_2(n,L_d)=0$ and $\St_n(L_d)\cong E_n(L_d)\cong L_d^\times$.
\end{corollary}

\begin{proof}
Since $K^\times=\{1\}$, we have $D_n(K)=\{I_n\}$. The proof of Theorem~\ref{thm:d-ary-all-units} therefore shows that $\GL_n(L_d)=E_n(L_d)$, and $\Theta_C$ identifies $E_n(L_d)$ with $L_d^\times$. The first two assertions follow from Theorem~\ref{thm:finite-presentation-reduction} and \eqref{eq:unstable-Schur-multiplier}.

Proposition~\ref{prop:stable-Ktwo-Leavitt} shows that $K_2(L_d)=0$. If $\kappa_n$ is injective, then $K_2(n,L_d)=0$. The extension \eqref{eq:unstable-Steinberg-extension} then identifies $\St_n(L_d)$ with $E_n(L_d)$.
\end{proof}

\begin{remark}\label{rem:unstable-Ktwo-obstruction}
The distinction between stable and unstable $K_2$ is essential here. The equality $K_2(L_d)=0$ asserts only that each element of $K_2(n,L_d)$ becomes trivial after finitely many stabilizations. It does not imply that $K_2(n,L_d)$ vanishes or is finitely generated.

For rings of finite Bass stable rank, the usual stability theorem identifies the unstable and stable groups in a suitable range. More precisely, if $\operatorname{sr}(R)=s+1<\infty$, then $K_2(n,R)\to K_2(R)$ is an isomorphism for $n\geq s+3$; see \cite[Chapter~III, Remark~5.5.2]{Weibel}. This criterion is unavailable here, since $\operatorname{sr}(L_d)=\infty$ by \cite[Theorem~2.8(2)]{AraPardoStableRank}. The isomorphisms $\Theta_C:M_n(L_d)\to L_d$ do not imply stability: they depend on $C$ and do not intertwine the standard matrix inclusions used to define $\kappa_n$.
\end{remark}


\begin{thebibliography}{99}

\bibitem{AbramsAnhPardo}
G. Abrams, P. N. \'{A}nh, and E. Pardo,
Isomorphisms between Leavitt algebras and their matrix rings,
\textit{J. Reine Angew. Math.} \textbf{624} (2008), 103--132.

\bibitem{AbramsArandaPino}
G. Abrams and G. Aranda Pino,
Purely infinite simple Leavitt path algebras,
\textit{J. Pure Appl. Algebra} \textbf{207} (2006), 553--563.

\bibitem{AraPardoStableRank}
P. Ara and E. Pardo,
Stable rank of Leavitt path algebras,
\textit{Proc. Amer. Math. Soc.} \textbf{136} (2008), 2375--2386.

\bibitem{AraBrustengaCortinas}
P. Ara, M. Brustenga, and G. Corti\~nas,
$K$-theory of Leavitt path algebras,
\textit{M\"unster J. Math.} \textbf{2} (2009), 5--34.

\bibitem{AraGoodearlPardo}
P. Ara, K. R. Goodearl, and E. Pardo,
$K_0$ of purely infinite simple regular rings,
\textit{$K$-Theory} \textbf{26} (2002), 69--100.

\bibitem{BrownloweSorensen}
N. Brownlowe and A. P. W. S\o rensen,
Leavitt $R$-algebras over countable graphs embed into $L_{2,R}$,
\textit{J. Algebra} \textbf{454} (2016), 334--356.

\bibitem{Freeland}
R. L. Freeland,
\textit{Relating Thompson's group $V$ to graphs of groups and Hecke algebras},
Ph.D. thesis, University of Cambridge, 2019,
doi:10.17863/CAM.52134.

\bibitem{KrsticMcCool}
S. Krsti\'c and J. McCool,
Presenting $\GL_n(k\langle T\rangle)$,
\textit{J. Pure Appl. Algebra} \textbf{141} (1999), 175--183.

\bibitem{LopatkinNam}
V. Lopatkin and T. G. Nam,
On the homological dimensions of Leavitt path algebras with coefficients in commutative rings,
\textit{J. Algebra} \textbf{481} (2017), 273--292.

\bibitem{Pardo}
E. Pardo,
The isomorphism problem for Higman--Thompson groups,
\textit{J. Algebra} \textbf{344} (2011), 172--183.

\bibitem{Weibel}
C. A. Weibel,
\textit{The $K$-book: An introduction to algebraic $K$-theory},
Graduate Studies in Mathematics, vol.~145,
American Mathematical Society, Providence, RI, 2013.

\end{thebibliography}
\end{document}